\documentclass[a4paper,12pt]{article}

\newenvironment{proof}{\noindent {\bf Proof:}}{\hfill $\Box$}

\newtheorem{theorem}{Theorem}
\newtheorem{lemma}{Lemma}

\newtheorem{corollary}{Corollary}

\newtheorem{assumption}{Assumption}
\newtheorem{remark}{Remark}

\usepackage{booktabs}

\usepackage{amsmath}
\usepackage{amssymb}
\usepackage{amsfonts}
\usepackage{graphicx}

\title{\bf Controller design and region of attraction estimation for nonlinear dynamical systems}

\begin{document}

\author{Milan Korda$^1$, Didier Henrion$^{2,3,4}$, Colin N. Jones$^1$}

\footnotetext[1]{Laboratoire d'Automatique, \'Ecole Polytechnique F\'ed\'erale de Lausanne, Station 9, CH-1015, Lausanne, Switzerland. {\tt \{milan.korda,colin.jones\}@epfl.ch}}
\footnotetext[2]{CNRS; LAAS; 7 avenue du colonel Roche, F-31400 Toulouse; France. {\tt henrion@laas.fr}}
\footnotetext[3]{Universit\'e de Toulouse; LAAS; F-31400 Toulouse; France.}
\footnotetext[4]{Faculty of Electrical Engineering, Czech Technical University in Prague,
Technick\'a 2, CZ-16626 Prague, Czech Republic.}

\date{Draft of \today}

\maketitle

\begin{abstract}
This work presents a method to obtain inner and outer approximations of the region of attraction of a given target set as well as an admissible controller generating the inner approximation. The method is applicable to constrained polynomial dynamical systems and extends to trigonometric and rational systems. The method consists of three steps: compute outer approximations, extract a polynomial controller while guaranteeing the satisfaction of the input constraints, compute inner approximations with respect to the closed-loop system with this controller. Each step of the method is a convex optimization problem, in fact a semidefinite program consisting of minimizing a linear function subject to linear matrix inequality (LMI) constraints. The inner approximations are positively invariant provided that the target set is included in the inner approximation and/or is itself invariant. 
\end{abstract}


\begin{center}\small
\end{center}

\section{Introduction}
In this paper, the region of attraction (ROA) of a given target set is defined as the set of all states that can be steered to the target set at any time while satisfying state and control input constraints. The problem of ROA characterization and computation with its many variations have a long tradition in both control and viability theory (where the ROA is typically called the capture basin~\cite{aubinViable}). Computational methods start with the seminal work of Zubov~\cite{zubov} and are surveyed in, e.g., \cite{roa,mci_outer,chesiPaper} and the book \cite{chesi}.

This work proposes a computationally tractable method for obtaining both inner and outer approximations of the ROA and, importantly, an approximate polynomial controller. The approach consists of three steps: compute outer approximations, extract a polynomial controller satisfying the input constraints, compute inner approximations of the closed-loop system with the extracted controller. Thus, the approach can also be viewed as a design tool providing a polynomial controller with an analytically known inner and outer approximations of its ROA.

Computing both inner and outer approximations, compared to just one or the other, enables assessing the tightness of the approximations obtained and provides a valuable insight into achievable performance and/or safety of a given constrained control system. For instance, a natural application for outer approximations is in collision avoidance, whereas a typical application of a (positively invariant) inner approximation is as a terminal constraint of a model predictive controller ensuring recursive feasibility of the underlying optimization problem~\cite{mayne}. 


The approach builds on and extends the ideas of~\cite{roa} and~\cite{mci_outer} and the controller extraction procedure of~\cite{anirudha} which was also sketched earlier in \cite{didier_switching} in the context of switching sequence design.
The main contributions with respect these works are:
\begin{itemize}
\item Contrary to~\cite{roa}, we treat the infinite time version of the ROA computation problem. The approach uses discounting similarly to our previous work~\cite{mci_outer} but here we treat the problem of computing the ROA, not maximum controlled invariant set as in~\cite{mci_outer}.
\item Contrary to~\cite{roa_inner_nolcos} we compute inner approximations for \emph{controlled} systems. This significantly extends the applicability of the approach but brings additional practical and theoretical challenges. In addition, under mild conditions, the inner approximations obtained are controlled invariant.
\item Contrary to~\cite{anirudha} the extracted controller is guaranteed to satisfy the input constraints and comes with an explicit estimate of its ROA, both from inside and outside.
\item The formulation providing outer approximations of the ROA is based on a different idea than that of~\cite{roa,mci_outer,anirudha} and provides tighter estimates on the numerical examples investigated. 
\end{itemize}

\looseness-1 As in previous works~\cite{roa,roa_inner_nolcos,mci_outer}, the method presented in this paper studies how whole ensembles of trajectories evolve through time using the concept of occupation measures. To obtain the outer approximations, we first characterize the ROA as a value function of a certain nonlinear optimal control problem which we then relax using measures in the spirit of~\cite{sicon}. This leads to a primal infinite-dimensional linear program (LP) which is then relaxed using a hierarchy of finite-dimensional semidefinite programming problems (SDPs) whose solutions can be used to extract approximate polynomial controllers. Finite-dimensional approximations of the dual infinite-dimensional LP in the space of continuous functions are sum-of-squares (SOS) problems and provide outer approximations to the ROA. To obtain the inner approximations, we characterize directly using measures the \emph{complement} of the ROA associated with the closed-loop system with the extracted polynomial controller. This leads to an infinite-dimensional primal LP in the space of measures. Finite-dimensional approximations of the dual LP on continuous functions are SOS problems and provide outer approximations to the complement of the ROA and hence inner approximations to the ROA itself.


Note in passing that
the use of occupation measures has a long tradition both in deterministic and stochastic control; see, e.g., \cite{mci_outer} for a historical account with the emphasis on applications to ROA and MCI set computation.


The paper is organized as follows. Section~\ref{sec:probStatement} defines the problem to be solved; Section~\ref{sec:occupMeas} introduces the occupation measures; Section~\ref{sec:outer} presents the outer approximation formulation; Section~\ref{sec:extraction} describes the controller extraction procedure; Section~\ref{sec:inner} presents the inner approximations; and Section~\ref{sec:numEx} demonstrates the whole procedure on numerical examples.

\subsection{Notation}
Throughout the paper we work with standard Euclidean spaces; all subsets of these spaces we refer to are automatically assumed Borel measurable. The spaces of continuous and once continuously differentiable functions on a set $X$ are denoted by $C(X)$ and $C^1(X)$, respectively. By a measure we understand a countably-additive mapping from sets to real numbers. Integration of a function $v(\cdot)$ with respect to a measure~$\mu$ over a set $X$ is denoted by $\int_X v(x)\,d\mu(x)$; often we omit the integration variable or the domain of integration and write $\int_X v\,d\mu$ or $\int v\,d\mu$. The support of a measure $\mu$ (i.e., the smallest closed set whose complement has a zero measure) is denoted by $\mathrm{spt}\,\mu$. A moment sequence $\{y_\alpha\}_{\alpha\in\mathbb{N}^n}$ of a measure $\mu$ on $\mathbb{R}^n$ is defined by $y_\alpha = \int_{\mathbb{R}^n} x^\alpha\,d\mu(x)=\int_{\mathbb{R}^n} x_1^{\alpha_1}\cdot \ldots \cdot x_n^{\alpha_n}\,d\mu(x)$. The indicator function of a set $A$, i.e., the function equal to one on the set and zero elsewhere, is denoted by $\mathbb{I}_A(\cdot)$.

\section{Problem description}\label{sec:probStatement}
Consider the polynomial input-affine dynamical system
\begin{equation}\label{eq:sys}
\dot{x}(t) = f(x(t)) + G(x(t))u(t),
\end{equation}
where the vector- and matrix-valued functions $f : \mathbb{R}^n\to \mathbb{R}^n$ and $G:\mathbb{R}^n\to\mathbb{R}^{n\times m}$ have polynomial entries. The system is subject to a basic semialgebraic state constraint\footnote{The assumption that $X$ is given by a super-level set of single polynomial is made for ease of exposition; all results extend immediately to arbitrary basic semialgebraic sets. This extension is briefly described in Appendix A.}
\begin{equation}\label{eq:X}
x(t) \in X := \{x\in\mathbb{R}^n\, :\, g_X(x) > 0 \},
\end{equation}
where $g_X$ is a polynomial, and box input constraints
\begin{equation}\label{eq:InputConst}
u(t) \in U := [0,\bar{u}]^m, \quad \bar{u} \ge 0.
\end{equation}
The assumption that the input constraint is of the form~(\ref{eq:InputConst}) is made without loss of generality since any box in $\mathbb{R}^m$ can be affinely transformed\footnote{Any box in $\mathbb{R}^m$ can, of course, be also affinely transformed to $[0,1]^m$. However, we decided to consider the more general form $[0,\bar{u}]^m$ so that it is immediately apparent where the upper bound~$\bar{u}$ comes into play in the optimization problems providing the region of attraction estimates.} to $[0,\bar{u}]^m$. 
It is also worth mentioning that arbitrary polynomial dynamical systems of the form $\dot{x} = f(x,u)$ can also be handled by considering the dynamic extension
\[
\begin{bmatrix} \dot{x} \\ \dot{u} \end{bmatrix} = \begin{bmatrix}f(x,u) \\ v \end{bmatrix}
\]
where the real control input $u$ is treated as a state and $v$ is a new, unconstrained, control input. Some of our convergence results hinge on the compactness of the input constraint set and therefore one may want to impose additional bounds on the new control input $v$, which correspond to slew-rate constraints on the true control input $u$, a requirement often encountered in practice.

In the remainder of the text we make the following standing assumption:
\begin{assumption}
The set $\bar{X} := \{x\in\mathbb{R}^n\, :\, g_X(x) \ge 0 \}$ is compact.
\end{assumption}
This assumption is of a technical nature, required for the convergence results of Section~\ref{sec:inner}.

Given a target set
\begin{equation*}
X_T := \{x\in\mathbb{R}^n\, :\, g_T(x) > 0 \} \subset X,
\end{equation*}
where the function $g_T$ is a polynomial, the goal of the paper is to compute inner approximations of the \emph{region of attraction} (ROA)
\begin{align*}
X_0 = \big\{x_0\in\mathbb{R}^n :\;\exists\, &u(\cdot), x(\cdot), \tau\in[0,\infty)\;\: \mathrm{s.t.}\;\: \dot{x}(t) = f(x(t)) + G(x(t))u(t) \;\mathrm{a.e.}, \\  &x(t)\in X, u(t)\in U,\:\forall\, t\in [0,\tau],\: x(0) = x_0,\: x(\tau)\in X_T   \big\},
\end{align*}
where a.e. means ``almost everywhere'' with respect to the Lebesgue measure on $[0,\tau]$, $x(\cdot)$ is absolutely continuous and $u(\cdot)$ is measurable.

In words, the region of attraction $X_0$ is the set of all initial states that can be steered to the target set $X_T$ at any time $\tau \in [0,\infty)$ in an admissible way, i.e., without violating the state or input constraints.

Our approach to compute the inner and outer approximations to the ROA consists of three steps:
\begin{enumerate}
\item Compute an outer approximation to the ROA,
\item Extract a polynomial controller out of this outer approximation,
\item Compute an inner approximation for the closed-loop system with this controller.
\end{enumerate}
These steps are detailed in the rest of the paper.

\section{Occupation measures}\label{sec:occupMeas}
The key ingredient of our approach is the use of measures to capture the evolution of a family of the trajectories of a dynamical system starting from a given initial distribution. 

Assume therefore that the initial state is not a single point but that its spatial distribution is given by an \emph{initial measure} $\mu_0$ and that to each initial condition a stopping time $\tau(x_0)\in [0,\infty]$ is assigned. Assume that the support of $\mu_0$ and the stopping time $\tau(\cdot)$ are chosen such that there exists a controller $u(x)$ such that all closed-loop trajectories $x(\cdot \!\mid\! x_0)$ starting from initial conditions $x_0\in \mathrm{spt}\,\mu_0$ remain in $X$ for all $t\in [0,\tau(x_0))$.

Then we can define the \emph{(average) discounted occupation measure} as
\begin{equation}\label{eq:mu}
\mu(A) = \int_X \int_0^{\tau(x_0)} e^{-\beta t}\mathbb{I}_A(x(t\mid x_0))\,dt\,d\mu_0(x_0),\quad A\subset X,
\end{equation}
where $\beta>0$ is a discount factor. This measure measures the average (where the averaging is over the distribution of the initial state) discounted time spent in subsets of the state-space in the time interval $[0,\tau(x_0))$.

The \emph{discounted final measure} $\mu_T$ is defined by
\begin{equation}\label{eq:mu_T}
\mu_T(A) = \int_X e^{-\beta\tau(x_0)}\mathbb{I}_A(x(\tau(x_0)\mid x_0))\,d\mu_0(x_0),\quad A\subset X,
\end{equation}
where we define $e^{-\beta \tau(x_0)} := 0$ whenever $\tau(x_0) = +\infty$. The discounted final measure captures the time-discounted spatial distribution of the state at the stopping time $\tau(x_0)$.

The equation linking the three measures is a variant of the discounted Liouville equation
\begin{equation}\label{eq:discountLiouville}
\int_X v(x) d\mu_T(x) + \beta\int_X  v(x)\, d\mu(x) =
 \int_X v(x)\,d\mu_0(x) + \int_X \nabla\, v(x) \cdot [f(x) + G(x)u(x)]\,d\mu(x),
\end{equation}
which holds for all $ v\in C^1(X)$. This equation will replace the system dynamics~(\ref{eq:sysUncont}) when studying the evolution of trajectories starting from the initial distribution $\mu_0$ over the possibly infinite time intervals~$[0,\tau(x_0))$. The equation is derived in Appendix~B.

\section{Outer approximations}\label{sec:outer}
In this section we formulate an infinite-dimensional linear program (LP) in the space of measures characterizing the ROA (or more precisely a set closely related to it) and whose finite-dimensional relaxations allow for the extraction of an approximate controller. Finite-dimensional relaxations of the dual to this LP then provide outer approximations to the ROA.

Here we take an approach different\footnote{The works~\cite{roa,mci_outer,anirudha} do not use the value function of an optimal control problem and argue directly in the space of trajectories modelled by measures. This approach can also be used to characterize the infinite-time ROA sought here; however, on our class of numerical examples, the tightness of the finite-dimensional relaxations and the quality of the extracted controllers seem to be inferior to those of the approach presented here. Note that this is something peculiar to the infinite-time ROA problem and has not been observed for the finite-time ROA problem of~\cite{roa,anirudha} and the maximum controlled invariant set problem of~\cite{mci_outer}.} from~\cite{roa,mci_outer,anirudha} and formulate the ROA computation problem as an optimal control problem whose value function characterizes the ROA; this problem is then relaxed using measures. The optimal control problem reads:

\begin{equation}\label{opt:opc_outer}
\begin{array}{rclll}
V(x_0) & := & \sup\limits_{\tau\in[0,\infty],u(\cdot)} &  e^{-\beta\tau}\big[\mathbb{I}_{X_T}(x(\tau))  -\mathbb{I}_{X\setminus X_T}(x(\tau))  \big] \\
&& \hspace{0.6cm}\mathrm{s.t.} & \dot{x}(t) = f(x(t)) + G(x(t))u(t),\quad x(0) = x_0\\
&&& x(t)\in X,\; u(t)\in U \; \forall\, t\in [0,\tau],
\end{array}
\end{equation}
where the objective function is zero if $\tau = +\infty$. Clearly, the initial conditions that can be steered to $X_T$ achieve a strictly positive objective; the initial conditions that necessarily leave $X$ a strictly negative objective; and the initial conditions that can be kept within $X$ forever but do not enter $X_T$ achieve a zero objective. Therefore $X_0 =  \{x\mid V(x) > 0\}$.

Now we formulate an infinite-dimensional LP relaxation of the problem (\ref{opt:opc_outer}) in the space of measures. First, and this the key insight due to~\cite{anirudha} ensuring that an approximate controller can be extracted, we view each component $u_i(x)$ of the controller $u(x)$ in~(\ref{eq:discountLiouville}) as a density of a \emph{control measure}\footnote{In \cite{anirudha} the authors use signed measures which they then decompose using the Jordan decomposition -- this step is avoided here since our control input is, without loss of generality, non-negative; this reduces the number of control measures by half and therefore makes the subsequent SDP relaxations more tractable.} $\sigma_i$ defined by
\begin{equation}\label{eq:sigmaDef}
\sigma_i(A) = \int_A u_i(x)\,d\mu(x),\quad i=1,2,\ldots,m,\quad A\subset X.
\end{equation}

Defining the linear differential operators $\mathcal{L}_f:C^1(X)\to C(X)$ and $\mathcal{L}_{G_i}:C^1(X)\to C(X)$, $i=1,\ldots,m$, by
\[
\mathcal{L}_fv := \nabla v \cdot f,\quad \mathcal{L}_{G_i}v := \nabla v \cdot G_i,
\]
where $G_i$ denotes the $i^{\mathrm{th}}$ column of $G$, the Liouville equation~(\ref{eq:discountLiouville}) can be rewritten as
\begin{equation}\label{eq:Liouville}
\int_{X} v\, d\mu_T + \beta\int_X v\,d\mu   =\int_X v\,d\mu_0 +  \int_X \mathcal{L}_f v\,d\mu + \sum_{i=1}^m\int_X  (\mathcal{L}_{G_i}v)\, d\sigma_i,
\end{equation}
where we used the fact that $d\sigma_i(x) = u_i(x)d\mu(x)$ in view of (\ref{eq:sigmaDef}).

The input constraints $u_i(x) \in [0,\bar{u}]$ are then enforced by requiring that $0\le \sigma_i \le \bar{u}\mu$, which is equivalent to saying that $\sigma_i$ is a nonnegative measure absolutely continuous with respect to~$\mu$ with density (i.e., Radon-Nikod\'ym derivative) taking values in $[0,\bar{u}]$. The constraint $0\le \sigma_i\le \bar{u}\mu$ can be written equivalently as $\sigma_i\ge 0 $, $\sigma_i + \hat{\sigma}_i = \bar{u}\mu$ for some non-negative slack measure $\hat{\sigma}_i \ge 0$. The state constraint $x(t)\in X$ is handled by requiring that $\mathrm{spt}\,\mu \subset X$. Further we decompose the final measure as $\mu_T = \mu_T^1 + \mu_T^2$ with\footnote{We could have constrained the support of $\mu_T^2$ to $X\setminus X_T$ instead of $X$; however, this is unnecessary by virtue of optimality in~(\ref{rrlp_2}) and avoids using the difference of two basic semialgebraic sets which may not be basic semialgebraic.} $\mathrm{spt}\,\mu_T^1\subset X_T$ and $\mathrm{spt}\,\mu_T^2\subset X$. Finally we set the initial measure equal to the Lebesgue measure (i.e., $d\mu_0(x) = dx$).

This leads to the following infinite-dimensional primal LP on measures:
\begin{equation}\label{rrlp_2}
\begin{array}{rclll}
 p^*_{o} &= & \sup & \int1\,d\mu_T^1 - \int1\,d\mu_T^2 \\\vspace{0.6mm}
&& \hspace{-1.4cm}\mathrm{s.t.} &\hspace{-1.55cm}\beta\int_{X} v\, d\mu + \int_{X_T} v\, d\mu_T^1 + \int_{X} v\, d\mu_T^2 =\! \int_X v\,dx+\!  \int_X \!\mathcal{L}_f v\,d\mu + \sum_{i=1}^m\int_X  (\mathcal{L}_{G_i}v)\, d\sigma_i &\!\!\forall\,v\in C^1(X) \\\vspace{0.6mm}
&&& \hspace{-1.5cm}\int p_i \,d\sigma_i + \int p_i \,d\hat{\sigma}_i -  \bar{u}\int p_i\,d\mu=0 \hspace{3.11cm} \forall\: i\in\{1,\ldots,m\}\; &\!\!\forall\,p_i\in C(X) \\\vspace{0.6mm}
&&& \hspace{-1.5cm}\mathrm{spt}\:\mu \subset X, \:\:  \mathrm{spt}\:\mu_T^1 \subset X_T,\:\: \mathrm{spt}\:\mu_T^2 \subset X,\\\vspace{0.6mm}
&&& \hspace{-1.5cm}\mathrm{spt}\:\sigma_i \subset X,\: \mathrm{spt}\:\hat{\sigma}_i \subset X, & \hspace{-4.9cm} \forall\: i\in\{1,\ldots,m\}\\\vspace{0.6mm}
&&& \hspace{-1.5cm} \mu\geq 0,\: \mu_T^1\geq 0,\: \mu_T^2\geq 0,\\\vspace{0.6mm}
&&& \hspace{-1.5cm}\sigma_i\geq 0, \: \hat{\sigma}_i\geq 0,& \hspace{-4.9cm} \forall\: i\in\{1,\ldots,m\},\\
\end{array}
\end{equation}
where the supremum is over
\[
(\mu,\mu_T^1,\mu_T^2,\sigma_1,\ldots,\sigma_m,\hat{\sigma}_1,\ldots,\hat{\sigma}_m) \in  M(X)\times M(X_T) \times M(X)\times M(X)^m\times M(X)^m.
\]

The dual LP on continuous functions provides approximations from above to the value function and therefore outer approximations of the ROA. The dual LP reads

\begin{equation}\label{vlp_2}
\begin{array}{rclll}
d^*_o & = & \inf & \int_{X} v(x)\, dx \\
&& \mathrm{s.t.} & \mathcal{L}_fv(x) + \bar{u}\sum_{i=1}^m p_i(x) \leq \beta v(x), \:\: &\forall\, x \in X \\
&&& p_i(x) \ge \mathcal{L}_{G_i} v(x), \:\: &\forall\, x \in X, \;\;i\in\{1,\ldots,m\} \\
&&& p_i(x) \ge 0, \:\: &\forall\, x \in X, \;\;i\in\{1,\ldots,m\}\\
&&& v(x) \geq 1, \:\: &\forall\, x \in X_T\\
&&& v(x) \geq -1, \:\: &\forall\, x \in X
\end{array}
\end{equation}
where the infimum is over $(v,p_1,\ldots,p_m) \in C^1(X)\times C(X)^m$. 

The following Lemma shows that the zero super-level set of any function $v\in C^1$ feasible in~(\ref{vlp_2}) is an outer approximation to the ROA.

\begin{lemma}\label{lem:outer2}
If $v\in C^1(X)$ is feasible in~(\ref{vlp_2}), then $X_0 \subset \{x : v(x) > 0 \}$.
\end{lemma}
\begin{proof}
Fix an $x_0\in X_0$. Then by definition there exists $m$ control functions $u_i(\cdot\!\mid\! x_0)\in [0,\bar{u}]$ and a time $\tau \in [0,\infty)$ such that $x(\tau \!\mid\! x_0) \in X_T$. Therefore $x(t \!\mid\! x_0) \in X$ for all $t\in [0,\tau]$ and consequently, using the constraints of~(\ref{vlp_2}),
\begin{align*}
\frac{d}{dt}v(x(t \!\mid\! x_0)) &=  \mathcal{L}_f v(x(t\!\mid\! x_0)) + \sum_{i=1}^{m}  \mathcal{L}_{G_i} v(x(t\!\mid\! x_0)) u_i(x(t \!\mid\! x_0)) \le \mathcal{L}_f v + \sum_{i=1}^m p_i u_i\\ & \le \mathcal{L}_f v + \bar{u}\sum_{i=1}^m p_i \le \beta v(x(t\!\mid\! x_0))
\end{align*}
for all $t\in[0,\tau)$. Using the Gronwall's inequality we have $v(x(\tau \!\mid\! x_0)) \le e^{\beta \tau}v(x_0)$ and therefore
\[
v(x_0)\ge e^{-\beta \tau}v(x(\tau \!\mid\! x_0)) > 0
\]
as desired. Here the last inequality follows from the fact that $v(x) \ge 1$ on $X_T$.
\end{proof}

\subsection{Finite dimensional relaxations}\label{sec:FiniteDimRelax}
The infinite dimensional LPs~(\ref{rrlp_2}) and (\ref{vlp_2}) can be solved only approximately; a systematic way of approximating them is the so-called Lasserre hierarchy of semidefinite programming~(SDP) relaxations~\cite{lasserre}, originally introduced for static polynomial optimization and later extended to the dynamic case~\cite{sicon}.

Instead of optimizing over measures, this hierarchy takes only finitely many moments of the measures for the primal problem~(\ref{rrlp_2}) while imposing conditions necessary for these truncated moment sequences to be feasible in the primal LP via the so-called moment and localizing matrices. On the dual side, the function space is restricted to polynomials of a given degree while imposing sufficient conditions for the non-negativity via sum-of-squares conditions. One then refers to the relaxation of order $k$ when the first $2k$ moments of the measures are taken on the primal side and polynomials of total degree up to $2k$ on the dual side. Please refer to~\cite{roa} or \cite{anirudha} for more details on how to construct the relaxations for this particular problem or to~\cite{lasserre} for a general treatment.

Let $v_k(\cdot)$ be a polynomial of degree $2k$ solving the $k^{\mathrm{th}}$ order relaxation of the dual SDP~(\ref{vlp_2}). Then this polynomial is feasible in~(\ref{vlp_2}) and therefore, in view of Lemma~\ref{lem:outer2}, we can define the $k^{\mathrm{th}}$ order outer approximation of $X_0$ by
\begin{equation}\label{eq:outerApprox}
X_{0k}^{\mathcal{O}} := \{x\in X : v_k(x) > 0\}.
\end{equation}

The following theorem states that the running intersection $\cap_{i=1}^k X_{0i}^{\mathcal{O}}$ converges monotonically to the set no smaller than $\{x\mid V(x) \ge 0\}$, which is the union of the ROA $X_0$ and the set of all states which can be kept within $X$ forever. 
\begin{theorem}\label{thm:outer}
The following statements are true: $\cap_{i=1}^k X_{0i}^{\mathcal{O}}\supset X_0$ for all $k\ge 1$ and
\[
\lim_{k\to\infty}\mathrm{vol}\big(\cap_{i=1}^k X_{0i}^{\mathcal{O}} \setminus\{x\mid V(x) \ge 0\}\big) = 0.
\]
\end{theorem}
\begin{proof}
The proof follows the convergence of $v_k(\cdot)$ to $V(\cdot)$ in $L_1$ norm which can be established using similar reasoning as in~\cite[Theorem~6]{mci_outer}; details are omitted for brevity.
\end{proof}

Note that in the case where the volume of the set of states which can be kept in $X$ forever but cannot be steered to $X_T$ is positive, the set to which the running intersection of $X_{0k}^{\mathcal{O}}$ converges can be strictly larger than $X_0$ (in the sense of positive volume difference); nevertheless by virtue of optimality in~(\ref{rrlp_2}) the controller attaining the infimum in~(\ref{rrlp_2}) generates $X_0$ (i.e., admissibly steers any state in $X_0$ to $X_T$).

\section{Controller extraction}\label{sec:extraction}
In this section we describe how a polynomial controller approximately feasible in the primal~LP~(\ref{rrlp_2}) can be extracted from the solution to the finite-dimensional relaxations of~(\ref{rrlp_2}). Given a truncated moment sequence solving the $k^{\mathrm{th}}$ primal relaxation, the idea is to find, component-by-component, polynomial controllers $u_i^k(x)$, $i=1,\ldots,m$, of a predefined total degree $\mathrm{deg}(u_i^k)\le k$ that approximately satisfy the relation~(\ref{eq:sigmaDef}). Details of this procedure are described below.

First, note that satisfying relation~(\ref{eq:sigmaDef}) is equivalent to satisfying
\[
\int_X v(x)u_i^k(x)d\mu(x) = \int_X v(x)d\sigma_i(x)
\] 
for all polynomials\footnote{This follows from the compactness of the constraint set $X$ and the fact that the polynomials are dense in $C(X)$ w.r.t. the supremum norm on compact sets.} $v(\cdot)$ and therefore, by linearity, it is equivalent to the linear equation
\begin{equation}\label{eq:contExtract}
\int_X x^\alpha u_i^k(x)\,d\mu(x) = \int_X x^\alpha\,d\sigma_i(x),   
\end{equation} 
where the multindex $\alpha = (\alpha_1,\ldots,\alpha_n)\in\mathbb{N}^n$ runs over all nonnegative integer $n$-tuples. The data available to us after solving the $k^\mathrm{th}$ order primal relaxation are the first $2k$ moments\footnote{The first $2k$ moments belong to a measure feasible in~(\ref{rrlp_2}) only asymptotically, i.e. when $k\to\infty$; at a finite relaxation order $k$ the first $2k$ moments may not belong to a measure feasible in (\ref{rrlp_2}) (or any nonnegative measure at all).} of the measures $\mu$ and $\sigma_i$. The approach for controller extraction of \cite{anirudha} consists of setting $\mathrm{deg}(u_i^k) := k$ and satisfying the equation~(\ref{eq:contExtract}) exactly for all moments up to degree $k$, i.e., for all $\alpha\in\mathbb{N}^n$ such that $\sum_{i=1}^n\alpha_i \le k$. For this choice of $\mathrm{deg}(u_i^k)$, the linear equation~(\ref{eq:contExtract}) takes a particularly simple form
\begin{equation}\label{eq:contExtractExact}
M_k(\boldsymbol y^{2k})\boldsymbol{u}_i^k = \boldsymbol{\sigma}_i^k,
\end{equation}
where $M_k(\boldsymbol y^{2k})$ is the ${{n+k} \choose {n}} \times {{n+k}\choose n} $ moment matrix (see~\cite{lasserre} for definition) associated to
$\boldsymbol{y}^{2k}$, the vector of the first $2k$ moments of $\mu$, $\boldsymbol{\sigma}_i^k$ is the vector of the first~$k$ moments of~$\sigma_i$, and~$\boldsymbol{u}_i^k$ is the vector of the coefficients of the polynomial $u_i^k(\cdot)$.

However, this extraction procedure does not ensure the satisfaction of the input constraints and indeed typically leads to controllers violating the input constraints. One remedy is to pose input constraints on the polynomial $u_i^k(x)$ as sum-of-squares (SOS) constraints and minimize over~$\boldsymbol{u}_i^k $ the moment mismatch $\|M_k(\boldsymbol y^{2k})\boldsymbol{u}_i^k - \boldsymbol{\sigma}_i^k\|_2$ subject to these input constraints; this immediately translates to an SDP problem, as proposed originally in \cite{density}. Denoting by $\hat{u}_i^k$ the ``true'' controller, coefficients of which satisfy the equation~(\ref{eq:contExtractExact}), this approach is equivalent to minimizing the $L^2(\mu)$ error\footnote{The integral $\int_X (u_i^k(x) - \hat{u}_i^k(x))^2\,d\mu(x) $ should be understood symbolically, replacing moments of the ``true'' occupation measure $\mu$ by $\boldsymbol y^{2k}$.} $\int_X (u_i^k(x) - \hat{u}_i^k(x))^2\,d\mu(x) $. The problem of this extraction procedure is now apparent --  the $L^2(\mu)$ criterion weights subsets of $X$ according to the occupation measure $\mu$ and therefore will penalize more those subsets of $X$ where the trajectories spend a large amount of time and penalize less those subsets where the trajectories spend little time. This is clearly undesirable and is likely to lead to a poor closed-loop performance of the extracted controller. What we would like to have is a uniform penalization over the constraint set $X$ or even better over the region of attraction $X_0$ (since any control is futile outside $X_0$ anyway). This leads to the following two-step procedure that we propose:
 \begin{enumerate}
 \item Extract the controller polynomial $\hat{u}_i^k$ satisfying exactly~(\ref{eq:contExtractExact}) but violating the input constraints.
 \item Minimize the $L^2(\nu)$ error $\int_X (u_i^k(x) - \hat{u}_i^k(x))^2\,d\nu(x) $ subject to the input constraints,
 \end{enumerate}
 where $\nu$ is a reference measure, ideally equal to the uniform measure on the ROA $X_0$. The second step can be carried out only approximately since the ROA, let alone the uniform distribution on it, are not known in advance. Here we use the first $2k$ moments of the uniform measure on~$X$ as a rough proxy for the uniform measure on~$X_0$.
This leads to the following SOS optimization problem for controller extraction:
\begin{equation}\label{contSOS}
\begin{array}{rclll}
& \min\limits_{ u_i^k,s_1^i,s_2^i}& \boldsymbol{u}_i^kM_1\boldsymbol{u}_i^k - 2\boldsymbol{u}_i^kM_2\boldsymbol{\hat{u}}_i^k \\
& \mathrm{s.t.} & u_i^k - g_X s_1^{i} = \text{SOS} \\
&&  \bar{u} - u_i^k - g_X s_2^i = \text{SOS}\\
&&  s_1^i = \text{SOS}, \;\; s_2^i = \text{SOS},\\
\end{array}
\end{equation} 
where $M_1$ and $M_2$ are respectively the ${ {n+\mathrm{deg}(u_i^k)} \choose  n } \times { {n+\mathrm{deg}(u_i^k)} \choose  n }$ and ${ {n+\mathrm{deg}(u_i^k)} \choose  n } \times { {n+k} \choose  n }$ top-left sub-blocks of $M_k(\boldsymbol y_{\nu}^{2k})$, the moment matrix associated to the first $2k$ moments of the reference measure~$\nu$. The optimization is over the coefficients $\boldsymbol{u}_i^k$ of the controller polynomial $u_i^k(x)$ and the coefficients of the SOS polynomial multipliers $s_1^i(x)$ and~$s_2^i(x)$.

It is worth noting that if a low-complexity controller is desired one can enforce sparsity on the coefficient vector of the polynomial $u_i^k$ by adding an $l_1$-regularization term $\gamma ||\boldsymbol{u}_i^k||_1$ for some $\gamma > 0$ to the objective of~(\ref{contSOS}) and/or fix an a priori sparsity pattern of $\boldsymbol{u}_i^k$.


To state a convergence result for the extracted controllers, we abuse notation and set $\int v(x) d\mu_k := \sum_\alpha v_\alpha  \boldsymbol y^{2k}(\alpha)$ for a polynomial $v(x) = \sum_\alpha v_\alpha x^\alpha $ of total degree less than $2k$. 

The following convergence result  for the controllers $\hat{u}_i^k$ was established in~\cite{anirudha}:
\begin{lemma}\label{lem:contConv}
There exists a subsequence $\{k_{j}\}_{j=1}^\infty$ such that
\[
\lim_{j\to\infty} \int_X v(x)  \hat{u}_i^{k_j}(x) \,d\mu_{k_j}(x) = \int_X v(x)  \hat{u}_i^*(x) \,d\mu^*(x),
\]
for all polynomials $v(x)$, where $u_i^*(x)d\mu^*(x) = d\sigma_i^*(x)$ and $(\mu^*, \sigma_i^*)$ is a part of an optimal solution to the primal LP~(\ref{rrlp_2}). 
\end{lemma}

Establishing a stronger notion of convergence and extending the convergence result to the controllers $u_i^k(x)$ satisfying the input constraints obtained from~(\ref{contSOS}) is currently investigated by the authors.

\begin{remark}
The method of controller extraction ensuring the satisfaction of input constraints using sum-of-squares programming is not the only one possible but is particularly convenient for subsequent inner approximation computation since the extracted controller is polynomial. For instance, other viable approach is to simply clip the controller $\hat{u}_i^k(x)$ on the input constraint set; in this case the closed-loop dynamics is piecewise polynomial defined over a semi-algebraic partition and it is still possible to compute inner approximations by a straightforward modification of the approach of Section~\ref{sec:inner}. 
\end{remark}

\section{Inner approximations}\label{sec:inner}
Given a controller $u(x)$ satisfying input constraints extracted from the outer approximations, all we need to do to obtain \emph{inner approximations} to the ROA $X_0$ is to compute inner approximations of the ROA for the closed-loop system
\begin{equation}\label{eq:sysUncont}
\dot{x} = \bar{f}(x) : = f(x) + G(x)u(x).
\end{equation}

\begin{remark}
All results of this section apply to uncontrolled systems $\dot{x}=\bar{f}(x)$ with an arbitrary polynomial vector field $\bar{f}$, not only those of the special form~(\ref{eq:sysUncont}).
\end{remark}


In order to compute the inner approximations we combine ideas of our two previous works~\cite{roa_inner_nolcos} and \cite{mci_outer} for computation of inner approximations to the ROA in a finite-time setting and outer approximations to the maximum controlled invariant set, respectively. This combination of the two approaches retains strong theoretical guarantees of both and seems to exhibit faster convergence of the finite-dimensional SDP relaxations.

The key idea of~\cite{roa_inner_nolcos} that we adopt here is to characterize the \emph{complement} of the ROA
\[
X_0^c := X\setminus X_0.
\]
By continuity of solutions to~(\ref{eq:sysUncont}), this set is equal to
\begin{align*}
X_0^c = \big\{x_0\in X \: : \: &\exists\,x(\cdot)\;\text{s.t.}\; \dot{x}(t)=\bar{f}(x(t))\; \text{and}\;\\
 & \exists\, \tau\in [0,\infty)\; \mathrm{s.t.}\; x(\tau)\in X_\partial \ \text{and/or}\ x(t)\in X_T^c\:\forall\,t\in[0,\infty)\big\},
\end{align*}
where
\[
 X_T^c := \{x\in\mathbb{R}^n \: : \: g_X(x)\ge 0,\; g_T(x)\le0\}
\]
is the complement of $X_T$ in $X$ and
\[
 X_\partial := \{x\in\mathbb{R}^n\::\: g_X(x) = 0 \}.
\]

In order to compute outer approximations of the complement ROA $X_0^c$ we study families of trajectories starting from an initial distribution $\mu_0$. The time evolution is captured by the discounted occupation measure~(\ref{eq:discountLiouville}) and the distribution at the stopping time $\tau(\cdot) \in [0,\infty]$ by the final measure $\mu_T$~(\ref{eq:mu_T}). The three measures are again linked by the discounted Liouville's equation~(\ref{eq:discountLiouville}).

The complement ROA $X_0^c$ (and hence also the ROA $X_0$) is then obtained by maximizing the mass of the initial measure subject to the discounted Liouville equation~(\ref{eq:discountLiouville}), support constraints $\mathrm{spt}\, \mu\subset X_T^c$, $\mathrm{spt}\, \mu_0\subset X_T^c$, $\mu_T\subset X_\partial$, and subject to the constraint that the initial measure is dominated by the Lebesgue measure (which is equivalent to saying that the density of the initial measure is below one). The last constraint is equivalent to the existence of a nonnegative slack measure $\hat{\mu}_0$ such that $\int_X w \,d\mu_0  + \int_X w \,d\hat{\mu}_0 = \int_X w\,dx$ for all~$w\in C(X)$. This optimization procedure yields an optimal initial measure with density (w.r.t. the Lebesgue measure) equal to one on $X_0^c$ and zero otherwise.

Writing the above formally leads to the following primal infinite-dimensional LP on measures
\begin{equation}\label{rrlp_inner}
\begin{array}{rclll}
 p^*_I &= & \sup & \int1\,d\mu_0 \\
&& \hspace{-1.4cm}\mathrm{s.t.} &\hspace{-1.55cm}\int_X v d\mu_T + \beta\int_X  v\, d\mu = \int_X v\,d\mu_0 + \int_X \nabla\, v \cdot \bar{f}\,d\mu, &\!\!\forall\,v\in C^1(X), \vspace{0.8mm} \\
&&& \hspace{-1.5cm}\int_X w \,d\mu_0 + \int_X w \,d\hat{\mu}_0 = \int_X w\,dx,\; &\!\!\forall\,w\in C(X), \vspace{0.4mm}\\
&&& \hspace{-1.5cm}\mathrm{spt}\:\mu \subset X_T^c, \:\: \mathrm{spt}\:\mu_0 \subset X_T^c,\:\: \mathrm{spt}\:\mu_T \subset X_\partial,\:\: \mathrm{spt}\:\hat{\mu}_0 \subset X_T^c, \vspace{0.4mm}\\
&&& \hspace{-1.5cm}\mu_0\geq 0, \: \mu\geq 0,\: \mu_T\geq 0,\: \hat{\mu}_0\geq 0,
\end{array}
\end{equation}
where the supremum is over the vector of nonnegative measures
\[
(\mu_0,\mu,\mu_T,\hat{\mu}_0) \in  M(X_T^c)\times M(X_T^c) \times M(X_\partial)\times M(X_T^c).
\] 


The dual infinite-dimensional linear program on continuous functions reads
\begin{equation}\label{vlp_inner}
\begin{array}{rclll}
d^*_I & = & \inf & \displaystyle\int_{X} w(x)\, dx \vspace{1.2mm}\\
&& \mathrm{s.t.} & \nabla v(x) \cdot \bar{f}(x) \leq \beta v(x), \:\: &\forall\, x\in X_T^c, \vspace{0.4mm}\\
&&& w(x) \ge v(x) + 1, \:\: &\forall\, x \in X_T^c, \vspace{0.4mm}\\
&&& w(x) \geq 0, \:\: &\forall\, x \in X_T^c, \vspace{0.4mm}\\
&&& v(x) \geq 0, \:\: &\forall\, x \in X_\partial,
\end{array}
\end{equation}
where the infimum is over the pair of functions $(v,w)\in C^1(X)\times C(X)$.

The following lemma establishes that the set $\{x\in X : v(x) < 0\}$ for any function $v \in C^1(X)$ feasible in~(\ref{vlp_inner}) provides an inner approximation to the ROA $X_0$.
\begin{lemma}\label{lem:inner}
If $v \in C^1(X)$ is feasible in~(\ref{vlp_inner}), then $\{x\in X : v(x) < 0\} \subset X_0$.
\end{lemma}
\begin{proof}
We will prove the contrapositive, i.e., that whenever $x_0\in X_0^c$, then $v(x)\ge 0$. For that we distinguish two cases.

First, assume that $x_0\in X_0^c$ and that $x(t)\in X$ for all $t\in [0,\infty)$. In that case necessarily also $x(t)\in X_T^c$ for all $t\in [0,\infty)$ and the first constraint of~(\ref{vlp_inner}) implies that $\frac{d}{dt}v(x(t))\le \beta v(x(t))$ for all $t\in [0,\infty)$. Using Gronwall's inequality, this implies that $v(x(t)) \le v(x_0)e^{\beta t}$ or $v(x_0) \ge e^{-\beta t}v(x(t))$ for all $t\in [0,\infty)$. Since $X_T^c$ is compact and $v$ continuous this implies that $v(x(t))$ is bounded and therefore necessarily $v(x_0) \ge 0$.

Second, assume that there exists a time $\tau \in [0,\infty)$ such that $x(\tau) \in X_\partial$. Then the second constraint of~(\ref{vlp_inner}) implies that $v(x(\tau)) \ge 0$, and therefore, using again the first constraint of~(\ref{vlp_inner}) and 
Gronwall's inequality, we get $0\le v(x(\tau)) \le e^{\beta \tau}v(x_0)$ and therefore $v(x_0)\ge e^{-\beta \tau}v(x(\tau))\ge 0$ as desired.
\end{proof}

\subsection{Choice of the discount factor $\beta$}

The LPs~(\ref{rrlp_inner}) and (\ref{vlp_inner}) depend on the discount factor $\beta > 0$ which is a free parameter. Theoretical results pertaining to the infinite-dimensional LPs (\ref{rrlp_inner}) and (\ref{vlp_inner}) and convergence guarantees of their finite-dimensional relaxations do not depend on the value of $\beta$ as long as it is strictly positive. However, the \emph{speed of convergence} and the \emph{quality} (i.e., the tightness) of the ROA estimates coming out of the finite-dimensional relaxations does depend on $\beta$.

This dependence can be exploited to speed-up the convergence of the finite-dimensional relaxations by observing that if a vector of measures $(\mu_0,\mu,\mu_T,\hat{\mu}_0)$ is feasible in the primal LP~(\ref{rrlp_inner}) with a given value of the discount factor $\beta > 0$, then for any other value of $\beta' > 0$ there must exist discounted occupation and terminal measures $\mu^{\beta'}$ and $\mu_T^{\beta'}$ such that the vector of measures $(\mu_0,\mu^{\beta'},\mu_T^{\beta'},\hat{\mu}_0)$ is also feasible in~(\ref{rrlp_inner}). Therefore, instead of considering a single value of $\beta$ we can define a vector $\boldsymbol{\beta} = (\beta_1,\ldots,\beta_{n_\beta})$, $\beta_i > 0$, and optimize over vectors of discounted occupation and terminal measures $\boldsymbol \mu = (\mu^1,\ldots,\mu^{n_\beta})$ and $\boldsymbol \mu_T = (\mu_T^1,\ldots,\mu_T^{n_\beta})$, components of which have to satisfy the discounted Liouville equation with \emph{one common} initial measure~$\mu_0$. This leads to the following modified primal LP
\begin{equation}\label{rrlp_inner_beta}
\begin{array}{rclll}
 p^*_I &= & \sup & \int1\,d\mu_0 \\
&& \hspace{-1.4cm}\mathrm{s.t.} &\hspace{-1.55cm}\int_X v d\mu_T^{i} + \beta_i\int_X  v\, d\mu^{i} = \int_X v\,d\mu_0 + \int_X \nabla\, v \cdot \bar{f}\,d\mu^{i}, & i\in\{1,\ldots,n_\beta\},\;\; \forall\,v\in C^1(X), \vspace{0.8mm} \\
&&& \hspace{-1.5cm}\int_X w \,d\mu_0 + \int_X w \,d\hat{\mu}_0 = \int_X w\,dx,\; & \hspace{3.1cm}\forall\,w\in C(X), \vspace{0.4mm}\\
&&& \hspace{-1.5cm}\mathrm{spt}\:\mu_0 \subset X_T^c, \:\: \mathrm{spt}\:\hat{\mu}_0 \subset X_T^c, \vspace{0.4mm}\\
&&& \hspace{-1.5cm}\mathrm{spt}\:\mu^{i} \subset X_T^c, \:\: \mathrm{spt}\:\mu_T^{i} \subset X_\partial,& i\in\{1,\ldots,n_\beta\},\vspace{0.0mm}\\
&&& \hspace{-1.5cm}\mu_0\geq 0, \: \hat{\mu}_0\geq 0 \vspace{0.4mm}\\
&&& \hspace{-1.5cm} \mu^{i}\geq 0,\: \mu_T^{i}\geq 0, & i\in\{1,\ldots,n_\beta\} \vspace{0.15mm}\\
\end{array}
\end{equation}
where the supremum is over the vector of nonnegative measures
\[
(\mu_0,\boldsymbol \mu,\boldsymbol \mu_T,\hat{\mu}_0) \in  M(X_T^c) \times M(X_T^c)^{n_\beta}\times  M(X_\partial)^{n_\beta}\times M(X_T^c).
\] 

As already remarked the optimal values of the infinite-dimensional primal LPs~(\ref{rrlp_inner}) and (\ref{rrlp_inner_beta}) are the same; however, the finite-dimensional SDP relaxations of~(\ref{rrlp_inner_beta}) are likely to converge faster than those of~(\ref{rrlp_inner}) (and will never converge slower provided that the $\beta$ used in (\ref{rrlp_inner}) is among the components of the vector $\boldsymbol \beta$ used in~(\ref{rrlp_inner_beta})).

For completeness we also state the dual infinite-dimensional LP on continuous functions where the function $v$ from~(\ref{vlp_inner}) is replaced by a vector of functions~$\boldsymbol v = (v_1,\ldots,v_{n_\beta})$:
\begin{equation}\label{vlp_inner_beta}
\begin{array}{rcllll}
d^*_I & = & \inf & \displaystyle\int_{X} w(x)\, dx \vspace{1.2mm}\\
&& \mathrm{s.t.} & \nabla v_i(x) \cdot \bar{f}(x) \leq \beta_i v_i(x), \:\: &\forall\, x\in X_T^c, & \;\; i\in\{1,\ldots,n_\beta\}, \vspace{0.4mm}\\
&&& w(x) \ge 1+\sum_{i=1}^{n_\beta}v_i(x), \:\: &\forall\, x \in X_T^c, \vspace{0.4mm}\\
&&& w(x) \geq 0, \:\: &\forall\, x \in X_T^c, \vspace{0.4mm}\\
&&& v_i(x) \geq 0, \:\: &\forall\, x \in X_\partial, & \;\; i\in\{1,\ldots,n_\beta\},
\end{array}
\end{equation}
where the infimum is over the vector of functions $(\boldsymbol v,w)\in C^1(X)^{n_\beta}\times C(X)$.

\begin{remark}\label{rem:lemmas}
For any functions $v_1,\ldots,v_{n_\beta}$ feasible in~(\ref{vlp_inner_beta}) the results of Lemmata~\ref{lem:inner} and~\ref{lem:tauInvar} and Corollaries~\ref{cor:invariance1} and \ref{cor:invariance2} hold with the function $v$ replaced by $\sum_{i=1}^{n_\beta}v_i$.
\end{remark}

\subsection{Finite-dimensional relaxations}
The infinite dimensional primal and dual LPs~(\ref{rrlp_inner}) and~(\ref{vlp_inner}) give rise to finite-dimensional SDP relaxations in exactly the same fashion as outlined in Section~\ref{sec:FiniteDimRelax} or described in more detail in, for instance, \cite[Section~VI]{roa}. Further details are omitted for brevity.

Let $v^{k}_i$, $i=1,\ldots,n_\beta$, and $w_k$ denote the polynomials of degree $2k$ solving the $k^{\mathrm{th}}$ order SDP relaxation of the dual LP~(\ref{vlp_inner_beta}) and let
\begin{equation}\label{eq:innerApprox}
X_{0k}^{\mathcal{I}}:= \{x\in X : \sum_{i=1}^{n_\beta} v_i(x) < 0\}
\end{equation}
denote the $k^{\mathrm{th}}$ order inner approximation (since $X_{0k}^{\mathcal{I}} \subset X_0$ by Lemma~\ref{lem:inner} and Remark~\ref{rem:lemmas}).

The following theorem summarizes convergence properties of the finite-dimensional relaxations. Let $X_0^{\mathrm{cl}} \subset X_0$ be the \emph{closed-loop} ROA associated to the closed-loop system~(\ref{eq:sysUncont}) (defined analogously as ROA~$X_0$).
\begin{theorem}\label{thm:convInner}
The following convergence properties hold:
\begin{itemize}
\item The optimal values of the primal and dual SDP relaxations converge to the volume of $X\setminus X_0^{\mathrm{cl}} $ as the relaxation order $k$ tends to infinity.
\item The functions $w_k$ converge in $L^1$ to the indicator function of the set $X\setminus X_0^{\mathrm{cl}} $.
\item The sets  $X_{0k}^{\mathcal{I}} $ converge from inside to $X_0^{\mathrm{cl}} $ in the sense that the volume of $X_0^{\mathrm{cl}} \setminus X_{0k}^{\mathcal{I}}$ tends to zero as the relaxation order $k$ tends to infinity.
\end{itemize}
\end{theorem}
\begin{proof}
The proof follows by the same arguments as Theorem~6, Corollary~7 and Theorem~8 of~\cite{roa_inner_nolcos} using Theorems~2 and 4 of~\cite{mci_outer} in place of Theorems~1 and~5 of \cite{roa_inner_nolcos}.
\end{proof}

\subsection{Invariance of the inner approximations}
In this section we investigate under what conditions the inner approximations $X_{0k}^{\mathcal{I}} $ are controlled invariant for the system~(\ref{eq:sys}) and positively invariant for the system~(\ref{eq:sysUncont}). Recall that a subset of $\mathbb{R}^n$ is called \emph{positively invariant} for an uncontrolled ODE if trajectories starting in the set remain in the set forever. Similarly a subset of $\mathbb{R}^n$ is called \emph{controlled invariant} for a controlled ODE if there exists an admissible control input such that the trajectories starting in the set remain in the set forever.

The following Lemma leads almost immediately to the characterization of the invariance of $X_{0k}^{\mathcal{I}} $; it says that trajectories starting in $X_{0k}^{\mathcal{I}} $ stay there until they reach the target set~$X_T$ for the closed-loop system~(\ref{eq:sysUncont}).

\begin{lemma}\label{lem:tauInvar}
If $x(0) \in X_{0k}^{\mathcal{I}} $, then $x(t) \in X_{0k}^{\mathcal{I}}$ for all $t\in [0,\tau]$, where $\tau = \inf \{s\in[0,\infty)\mid x(s)\in X_T \} < \infty $ is the first time that $x(t)$ reaches~$X_T$, and $x(t)$ is the solution to~(\ref{eq:sysUncont}).
\end{lemma}
\begin{proof}
If $x(0) \in X_T$ then there is nothing to prove. Assume therefore $x(0)\in X\setminus X_T \subset X_T^c$. By Lemma~\ref{lem:inner} (and its analogy for the problem~(\ref{vlp_inner_beta})) if $x(0) \in X_{0k}^{\mathcal{I}} $, then, by the definition of the ROA $X_0$, $\tau = \inf \{s\in[0,\infty)\mid x(s)\in X_T \}$ is finite and $x(t) \in X$ for all $t\in [0,\tau]$. This implies that $x(t)\in X_T^c$ for all $t\in [0,\tau)$ and therefore $\sum_{i=1}^{n_\beta} v_i(x(t)) < \sum_{i=1}^{n_\beta} v_i(x(0)) < 0$, where the first inequality follows from the first constraint of~(\ref{vlp_inner_beta}) and the Gronwall's inequality and the second one from the definition of $X_{0k}^{\mathcal{I}}$~(\ref{eq:innerApprox}). This proves the claim.
\end{proof}

The following two immediate Corollaries give conditions under which  $X_{0k}^{\mathcal{I}}$ is invariant.
\begin{corollary}\label{cor:invariance1}
If the target set $X_T$ is controlled~/~positively invariant for the system~(\ref{eq:sys})~/~(\ref{eq:sysUncont}), then so is the set $X_{0k}^{\mathcal{I}} \cup X_T$.
\end{corollary}
\begin{corollary}\label{cor:invariance2}
If  $\mathrm{closure}(X_T) \subset  X_{0k}^{\mathcal{I}}$, then $X_{0k}^{\mathcal{I}}$ is controlled~/~positively invariant for the system~(\ref{eq:sys})~/~(\ref{eq:sysUncont}).
\end{corollary}


\section{Numerical examples}\label{sec:numEx}
In this section we present numerical examples illustrating the approach. As a modeling tool we used Gloptipoly~\cite{glopti} which allows to model directly the primal problems on measures. The outer and inner approximations are then extracted from the dual variables provided by a primal-dual SDP solver (in our case MOSEK). Equivalently one can model the dual SOS problems (in our case using YALMIP~\cite{yalmip} or SOSOPT~\cite{sosopt}) and extract the primal moment vector (which is needed to obtain the approximate control law) as a part of the dual variables associated with the constraints of the dual~SDP.

\subsection{Nonlinear double integrator}
As our first example we consider the nonlinear double integrator
\begin{align*}
\dot{x}_1 &= x_2 + 0.1x_1^3\\
\dot{x}_2 &= 0.3u
\end{align*}
subject to the constraints $u\in [-1,1]$ and $x\in X := \{x: \|x\|_2 < 1.2\}$. As a terminal set we take a small neighbourhood of the origin $X_T := \{x: ||x||_2 < 0.1\}$. First we obtain two approximate polynomial controllers of degree four by solving the fourth order (i.e., considering moments up to total degree eight) SDP relaxation of the two primal LPs~ (\ref{rrlp_2}). Then we obtain inner approximations from the eighth order (corresponding to degrees of $w$ and $v_i$, $i=1,\ldots,5$, equal to 16) SDP relaxation of the dual LP~(\ref{vlp_inner_beta}) where we choose the vector $\boldsymbol \beta = (10,1,0.1,0.01,0.001)$; the inner approximations given by (\ref{eq:innerApprox}) are compared in Figure~\ref{fig:doubInteg} with an outer approximation~(\ref{eq:outerApprox}) obtained from solving an SDP relaxation of the dual LP~(\ref{vlp_2}) with degrees of $w$ and $v$ equal to 16. We can see that, for this example, both the inner and outer estimates are almost tight. Computation times of the SDP relaxations are reported in Table~\ref{tab:timesDoubleInteg}.


\begin{table}[h]
\footnotesize
\centering
\caption{\rm \footnotesize\looseness-1 Nonlinear double integrator -- Computation time comparison for different degrees $d = 2k$ of the polynomials in the SDP relaxations of the outer dual LP (\ref{vlp_2}) and the inner dual LP (\ref{vlp_inner_beta}). Reported is pure solver time of the MOSEK SDP solver (excluding Gloptipoly and Yalmip parsing time). Larger solve time for the inner approximations is because of the larger number of decision variables and constraints in (\ref{vlp_inner_beta}) since in~(\ref{vlp_inner_beta}) there is one polynomial $v_i$ associated to each of the five values of the discount factor~$\beta_i$. }\label{tab:timesDoubleInteg}\vspace{2mm}
\begin{tabular}{cccccc}
\toprule
$d$ & 6 & 8 & 10 & 12 & 16 \\\midrule
Inner & 0.42\,s & 0.67\,s & 1.03\,s &  1.61\,s & 4.41\,s  \\\midrule
Outer & 0.17\,s & 0.23\,s & 0.37\,s &  0.66\,s & 1.02\,s  \\
\bottomrule
\end{tabular}
\end{table}

\begin{figure*}[t!]
	\begin{picture}(140,196)
	
	\put(120,20){\includegraphics[width=80mm]{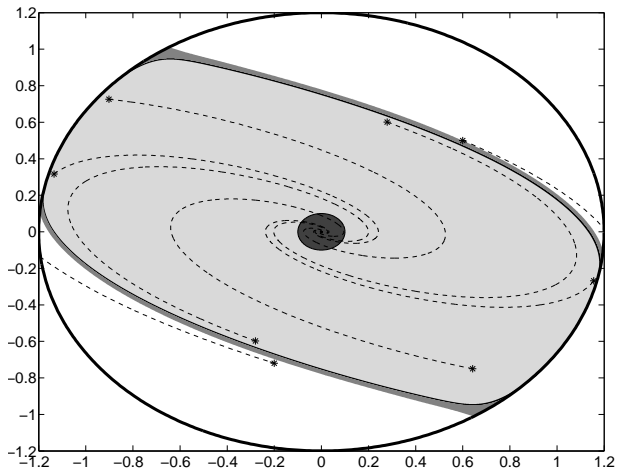}}

	\end{picture}
	\caption{\small Nonlinear double integrator -- light: inner approximation~(\ref{eq:innerApprox}) with $v_i$ of 16; darker: outer approximations~(\ref{eq:outerApprox}) with $v$ of degree 16; dark, small: target set; thick line: constraint set boundary; dashed thin lines: closed-loop trajectories.}
	\label{fig:doubInteg}
\end{figure*}

\subsection{Controlled 3D Van der Pol oscillator}
As our second example we consider a controlled Van der Pol oscillator in three dimensions given by
\begin{align*}
\dot{x}_1 &= -2x_2 \\
\dot{x}_2 &= 0.8x_1 - 2.1x_2 + x_3 + 10x_1^2x_2\\
\dot{x}_3 &= -x_3 + x_3^3 + 0.5u
\end{align*}
subject to the constraints $u\in [-1,1]$ and $x\in X = \{x : \|x\|_2 < 1\}$. As a terminal set we take a small neighbourhood of the origin $X_T = \{x : \|x\|_2 < 0.1\} $.
First we extract an approximate controller $u(x)$ of degree four by solving an SDP relaxation of fourth order (i.e., considering moments up to total degree eight) of the primal LP~(\ref{rrlp_2}). After that we compute an inner approximation by solving an SDP relaxation of the dual LP~(\ref{vlp_inner_beta}) with $\boldsymbol \beta = (1,0.1,0.01,0.001)$ and polynomials $w$ and $v_i$, $i=1,\ldots,4$, of degree 10. To assess the tightness of the inner approximation we compute an outer approximation by solving an SDP relaxation of the dual~LP~(\ref{vlp_2}) with $w$ and $v$ polynomials of degree 16. Figure~\ref{fig:VanDerPol3D} shows the comparison and also several trajectories of the closed-loop system which, as expected, converge to the target set (not shown) whenever starting in the inner approximation. We observe a relatively good tightness of the inner approximation. Computation times are reported in Table~\ref{tab:times3DVanDerPol}.

\begin{table}[h]
\footnotesize
\centering
\caption{\rm \footnotesize\looseness-1 Controlled 3D Van der Pol oscillator -- Computation time comparison for different degrees $d = 2k$ of the polynomials in the SDP relaxations of the outer dual LP~(\ref{vlp_2}) and the inner dual LP~(\ref{vlp_inner_beta}). The same comments as for Table~\ref{tab:timesDoubleInteg} apply.}\label{tab:times3DVanDerPol}\vspace{2mm}
\begin{tabular}{cccccc}
\toprule
$d$ & 6 & 8 & 10 & 12 & 16 \\\midrule
Inner & 1.16\,s & 3.6\,s & 12.3\,s &  32.7\,s & 213\,s  \\\midrule
Outer & 0.48\,s & 0.88\,s & 2.8\,s &  7.4\,s & 54.4\,s  \\
\bottomrule
\end{tabular}
\end{table}


\begin{figure*}[t!]
	\begin{picture}(140,246)
		\put(130,20){\includegraphics[width=70mm]{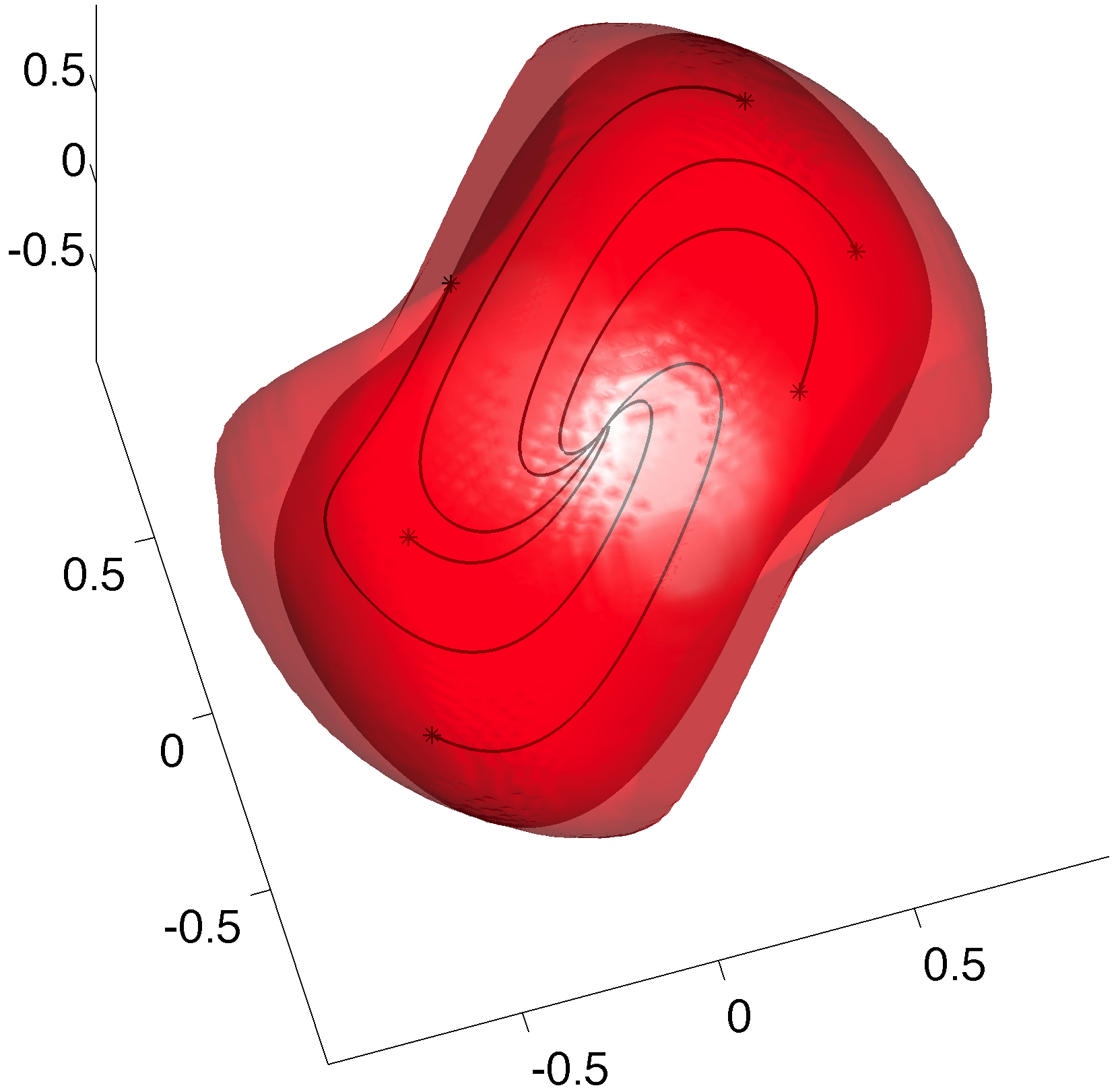}} 
	\end{picture}
	\caption{\small Controlled 3D Van der Pol oscillator -- inner approximation of degree 10 (dark, smaller); outer approximation of degree 16 (light, larger); closed-loop trajectories (black). The degree of the polynomial controller generating the inner approximation is four.}
	\label{fig:VanDerPol3D}
\end{figure*}

\section{Conclusion}
We have presented a method for computing inner and outer approximations of the region of attraction for input-affine polynomial dynamical systems subject to box input and semialgebraic state constraints. The method combines and extends the ideas of our previous works~\cite{roa,roa_inner_nolcos,mci_outer} and the controller extraction idea of~\cite{anirudha}. The inner approximations are controlled invariant provided that the target set is included in the inner approximation and/or itself controlled invariant.

The approach is based purely on convex optimization, in fact semidefinite programming (SDP), and is tractable for systems of moderate state-space dimension (say up to $n$ equal to $6$ or $8$) using interior point solvers, in our case MOSEK; the number of control inputs $m$ plays a secondary role in the computation complexity since the number of variables and constraints in the SDP relaxations grows only linearly with $m$. The linear growth is due to the fact that one measure on $\mathbb{R}^n$ is associated to each control input, where $n$ is the state dimension. The total number of variables therefore grows as $O(mn^d)$ when the polynomial degree is held fixed and as $O(m d^n)$ when $m$ and $n$ are held fixed; this is a significantly more favourable growth rate than $O((m+n)^d)$ or $O(d^{n+m})$ of \cite{roa,roa_inner_nolcos,mci_outer}. Larger systems could be tackled using first-order methods; for instance SDPNAL~\cite{sdpnal} shows promising results on our problem class (after suitable preconditioning of problem data).

The approach can be readily extended to trigonometric, rational or rational-trigonometric dynamics with the same theoretical guarantees since a variant of the Putinar Positivstellensatz~\cite{putinar} (which Theorem~\ref{thm:convInner} hinges on) holds for trigonometric polynomials as well~\cite{dumitrescu}. 

At present no proof of convergence of the inner approximations $X_{0k}^\mathcal{I}$ to the ROA $X_0$ is available. This is since Lemma~\ref{lem:contConv} provides only very weak convergence guarantees of the extracted controller to a controller generating the ROA $X_0$. Strengthening Lemma~\ref{lem:contConv} or the possibility of proving the convergence by other means is currently investigated by the authors.

\section*{Appendix A}

Here we describe how the presented approach can be extended to handle the situation where the state constraint set $X$ and/or the target set $X_T$ are given by multiple polynomial inequalities. Assume therefore for this section that
\[
X = \{x : g_X^i > 0,\: i=1,\ldots,n_X\}, \qquad X_T = \{x : g_T^i > 0, i=1,\ldots,n_T\}\subset X,
\]
and that the set
\[
\bar{X}  = \{x : g_X^i \ge 0,\: i=1,\ldots,n_X\}
\]
is compact. The primal and dual LPs~(\ref{rrlp_2}) and (\ref{vlp_2}) providing the outer approximations stay the same since the set $\bar{X} $ is compact basic semialgebraic. A slight modification is needed for the primal and dual LPs (\ref{rrlp_inner}) and (\ref{vlp_inner}) providing the inner approximations. There, with multiple constraints, we define the sets $X_\partial$ and $X_T^c$~as
\[
X_\partial = \bigcup_{i=1}^{n_X}X_\partial^i :=\bigcup_{i=1}^{n_X}\big\{x : g_X^i = 0,\: g_X^j \ge 0,\: j\in \{1,\ldots,n_X\}\setminus \{i\} \big\},
\]
\[
X_T^c = \bigcup_{i=1}^{n_T}X_T^{c^i} := \bigcup_{i=1}^{n_T}\big\{x : g_T^i \le 0,\: g_X^j \ge 0,\: j\in \{1,\ldots,n_X\} \big\} .
\]
Since $\bar{X}$ is compact, so are $X_\partial^i$ and $X_T^{c^i}$. The sets $X_\partial$ and $X_T^c$ are therefore unions of compact basic semialgebraic sets. In the primal LP the measures $\mu_0$, $\hat{\mu}_0$ and $\mu$ with the supports in~$X_T^c$ are decomposed as the sum of~$n_T$ measures each with the support in~$X_T^{c^i}$ and analogously the terminal measure with the support in~$X_\partial$ is decomposed as the sum of~$n_X$ measures with the supports in~$X_\partial^i$. In the dual LP~(\ref{vlp_inner}) this translates to imposing the inequalities for each $X_T^{c^i}$ and $X_\partial^i$; for instance, the first inequality of~(\ref{vlp_inner}) now translates to the~$n_T$ inequalities $\nabla v \cdot \bar{f}(x) \leq \beta v(x)\;\forall\: x\in X_T^{c^i}$, $i=1,\ldots, n_T$.


\section*{Appendix B}
This Appendix derives the discounted Liouville equation~(\ref{eq:discountLiouville}). For any test function $v \in C^1(X)$ we have
\begin{align*}
\int_X \nabla v(x) \cdot \bar{f}(x)\,d\mu(x) &= \int_X \int_0^{\tau(x_0)} e^{-\beta t}\, \nabla v(t\!\mid \! x_0)\cdot \bar{f}(x(t \!\mid\! x_0))\,dt\,d\mu_0(x_0)\\ 
& \hspace{-3.5cm}=\int_X \int_0^{\tau(x_0)} e^{-\beta t} \frac{d}{dt}v(x(t \!\mid\! x_0))\,dt\,d\mu_0(x_0)\\
&\hspace{-3.5cm}= \beta \int_X \int_0^{\tau(x_0)} \!\!\!\!\!\!e^{-\beta t}v(x(t \!\mid\! x_0))\,dt\,d\mu_0(x_0) +\!\! \int_X \!e^{-\beta \tau}v(x(\tau(x_0) \!\mid\! x_0))\,\mu_0(x_0) - \!\int_X\! v(x_0)\,\mu_0(x_0)\\
&\hspace{-3.5cm}= \beta \int_X v(x)\,d\mu(x) + \int_X v(x)\,d\mu_T(x) - \int_X v(x)\,d\mu_0(x),
\end{align*}
which is exactly~(\ref{eq:discountLiouville}). Here we have used integration by parts in the third equality, and the definition of the initial, discounted occupation and terminal measures in the fourth.


\end{document}